\documentclass[12pt,leqno]{article}
\usepackage{amsfonts}
\usepackage{amsmath, amsthm, amsfonts, amssymb, color}
\usepackage{mathrsfs}
\newtheorem{thm}{Theorem}[section]
\newtheorem{cor}[thm]{Corollary}

\newtheorem{prp}[thm]{Proposition}

\theoremstyle{definition}

\newcommand{\scr}[1]{\mathscr #1}
\definecolor{wco}{rgb}{0.5,0.2,0.3}

\numberwithin{equation}{section} \theoremstyle{remark}

\newcommand{\ua}{\uparrow}

\title{{\bf Hypercontractivity for Functional Stochastic Differential  Equations}\footnote{Supported in
 part by  Lab. Math. Com. Sys.,  NNSFC(11131003, 11431014) and the 985 project.}
}
\author{
{\bf  Jianhai Bao$^{b)}$,  Feng-Yu Wang$^{a),c)}$,  Chenggui Yuan$^{c)}$}\\
\footnotesize{$^{a)}$School of Mathematical Sciences, Beijing Normal
University, Beijing 100875, China}\\
\footnotesize{$^{b)}$School of Mathematics and Statistics, Central South
University, Changsha 410083, China}\\
 \footnotesize{$^{c)}$Department of Mathematics,
Swansea University, Singleton Park, SA2 8PP, UK}\\
\footnotesize{jianhaibao13@gmail.com, wangfy@bnu.edu.cn, F.-Y.Wang@swansea.ac.uk,
C.Yuan@swansea.ac.uk}}
\begin{document}
\def\R{\mathbb R}  \def\ff{\frac} \def\ss{\sqrt} \def\B{\mathbf
B}
\def\N{\mathbb N} \def\kk{\kappa} \def\m{{\bf m}}
\def\dd{\delta} \def\DD{\Dd} \def\vv{\varepsilon} \def\rr{\rho}
\def\<{\langle} \def\>{\rangle} \def\GG{\Gamma} \def\gg{\gamma}
  \def\nn{\nabla} \def\pp{\partial} \def\EE{\scr E}
\def\d{\text{\rm{d}}} \def\bb{\beta} \def\aa{\alpha} \def\D{\scr D}
  \def\si{\sigma} \def\ess{\text{\rm{ess}}}
\def\beg{\begin} \def\beq{\begin{equation}}  \def\F{\scr F}
\def\Ric{\text{\rm{Ric}}} \def\Hess{\text{\rm{Hess}}}
\def\e{\text{\rm{e}}} \def\ua{\underline a} \def\OO{\Omega}  \def\oo{\omega}
 \def\tt{\tilde} \def\Ric{\text{\rm{Ric}}}
\def\cut{\text{\rm{cut}}} \def\P{\mathbb P} \def\ifn{I_n(f^{\bigotimes n})}
\def\C{\scr C}      \def\aaa{\mathbf{r}}     \def\r{r}
\def\gap{\text{\rm{gap}}} \def\prr{\pi_{{\bf m},\varrho}}  \def\r{\mathbf r}
\def\Z{\mathbb Z} \def\vrr{\varrho} \def\l{\lambda}
\def\L{\scr L}\def\Tt{\tt} \def\TT{\tt}\def\II{\mathbb I}
\def\i{{\rm in}}\def\Sect{{\rm Sect}}\def\E{\mathbb E} \def\H{\mathbb H}
\def\M{\scr M}\def\Q{\mathbb Q} \def\texto{\text{o}} \def\LL{\Lambda}
\def\Rank{{\rm Rank}} \def\B{\scr B} \def\i{{\rm i}} \def\HR{\hat{\R}^d}
\def\to{\rightarrow}\def\l{\ell}\def\ll{\lambda}
\def\8{\infty}\def\ee{\epsilon}

\def\8{\infty}\def\ee{\epsilon} \def\Y{\mathbb{Y}} \def\lf{\lfloor}
\def\rf{\rfloor}\def\3{\triangle}\def\H{\mathbb{H}}\def\S{\mathbb{S}}\def\1{\lesssim}
\def\va{\varphi}

\maketitle

\begin{abstract}

An explicit sufficient condition on the hypercontractivity is derived for the Markov semigroup  associated to a   class of functional stochastic differential equations.  Consequently, the semigroup $P_t$ converges exponentially to its unique invariant probability measure $\mu$ in entropy,  $L^2(\mu)$ and the totally variational norm, and it is compact in $L^2(\mu)$ for large $t>0$. This   provides a natural class of  non-symmetric Markov semigroups which are compact for large time but non-compact for small time. A semi-linear model which may not satisfy this sufficient condition is also investigated. As  the associated Dirichlet form does not satisfy the log-Sobolev inequality, the standard argument  using functional inequalities does not work.

 \

\noindent
 {\bf AMS subject Classification:}\    65G17, 65G60    \\
\noindent {\bf Keywords:} Hypercontractivity, compactness, exponential ergodicity,   functional stochastic differential
equation, Harnack inequality.
 \end{abstract}

\section{Introduction}

The hypercontractivity, first found by Nelson \cite{Nelson} for the
Ornstein-Ulenbeck semigroup, has been investigated intensively for
various models of Markov semigroups, see   for instance \cite{Bakry,
Davies, Gross, Wbook1, Wbook2, Wbook3} and references within.
However, so far there is no any result  on this property for the
semigroup associated to functional stochastic differential equations
(FSDEs, or SDEs with memory).

It is well known by Gross (see \cite{Gross}) that the log-Sobolev inequality implies the hypercontractivity. However, for SDEs with delay the log-Sobolev inequality for the associated Dirichlet form does not hold.   Indeed, according to \cite[Theorem 3.3.6]{Wbook1},  the super Poincar\'e inequality (and hence the log-Sobolev inequality) implies the uniform integrability of the associated Markov semigroup $P_t$ for all $t>0$, which is not the  case for the Markov semigroup associated to SDEs with delay, since it is clear that in this case $P_t$ is not uniformly integrable for $t$ smaller than the length of time delay, see Remark 1.1(2) for details.

On the other hand,
  the dimension-free Harnack inequality introduced in
\cite{W97} and further developed in many other papers is a powerful
tool in the study of the hypercontractivity, which works well even
for non-linear SPDEs (see e.g. \cite{W07,L}). Recently,
this type Harnack inequalities   have been investigated in
\cite{WY} for FSDEs. To derive the hypercontractivity and
exponential ergodicity from  the dimension-free Harnack inequality, the key
point is to prove the Gauss-type  concentration property of the
unique invariant probability measure with respect to the uniform
norm on the state space, which is, however,
not easy   for FSDEs. We will see that our proof of  the
exponential integrability is tricky   (see the proof of Lemma
\ref{L2.1}).

\

Let $r_0>0$ be fixed, and let $\C:=C([-r_0,0];\R^d)$ be equipped
with the uniform norm $\|\cdot\|_\infty$. Let $\B_b(\C)$ be the set
of all bounded measurable functions on $\C$. Let $\{B(t)\}_{t\ge0}$
be a $d$-dimensional Brownian motion defined on $(\OO,\F,\{\F_t\}_{t\ge 0}, \P)$, a complete filtered
probability space. Let $\si$ be an
invertible $d\times d$-matrix, $Z\in C(\R^d;\R^d)$ and $b: \C\to
\R^d$ be Lipschitz continuous. Consider the following
FSDE on $\R^d$: \beq\label{1.1} \d X(t)= \big\{Z(X(t))+b(X_t)\big\}\d
t +\si\d B(t),\ \ X_0=\xi\in \C,\end{equation} where, for each $t\ge
0$, $X_t\in \C$ is fixed by $X_t(\theta):=X(t+\theta), \theta\in
[-r_0, 0].$ Assume that \beq\label{A1}
2\<Z(\xi(0))+b(\xi)-Z(\eta(0))-b(\eta), \xi(0)-\eta(0)\>\le \ll_2
\|\xi-\eta\|_\infty^2-\ll_1|\xi(0)-\eta(0)|^2,\ \
\xi,\eta\in\C\end{equation} holds for some constants
$\ll_1,\ll_2\ge 0.$  Then the equation (\ref{1.1}) has a unique strong
solution and the solution is non-explosive, see \cite[Theorem
2.3]{SR}. Let $P_t$ be the Markov semigroup associated to the
segment (functional) solution, i.e.
$$P_t f(\xi):= \E f(X_t^\xi),\ \ t\ge 0, f\in \B_b(\C), \xi\in \C,$$ where $X_t^\xi$ is the corresponding segment process of $X^\xi(t)$ which solves \eqref{1.1} for $X_0=\xi$.
The following is the first main result of the paper.

\beg{thm}\label{T1.1} If $\ll:= \sup_{s\in [0,\ll_1]} \big(s-\ll_2\e^{r_0s}\big)>0$, then $P_t$ has a unique invariant probability measure $\mu$, and the following assertions hold.
\beg{enumerate}\item[$(1)$] $P_t$ is hypercontractive, i.e. $\|P_t\|_{2\to 4}\le 1$ holds for large enough $t>0$, where $\|\cdot\|_{2\to 4}$ is the operator norm from $L^2(\mu)$ to $L^4(\mu)$.
\item[$(2)$] $P_t$ is compact on $L^2(\mu)$ for large enough $t>0$, and there exist constants $c,\aa>0$ such that 
$$\mu((P_tf)\log P_t f)\le c\e^{-\aa t}\mu(f\log f),\ \ \ t\ge 0, f\ge 0, \mu(f)=1.$$
\item[$(3)$] There exists a constant $C>0$ such that
$$\|P_t-\mu\|_2^2:= \sup_{\mu(f^2)\le 1}\mu\big((P_t f-\mu(f))^2\big)\le C\e^{-\ll t},\ \ t\ge 0.$$
\item[$(4)$] There exist two constants $t_0, C>0$ such that
$$\|P_t^\xi-P_t^\eta\|_{var}^2\le  C\|\xi-\eta\|_\infty^2 \e^{-\ll t},\ \ t\ge t_0,$$ where $\|\cdot\|_{var}$ is the total variational norm and $P_t^\xi$ stands for  the distribution of $X_t^\xi$ for   $(t,\xi)\in [0,\infty)\times\C$. \end{enumerate}\end{thm}

\paragraph{Remark 1.1} (1) It is easy to see that an invariant probability measure $\mu$ of $P_t$ is shift-invariant, that is, letting
$\phi_\theta(\xi)=\xi(\theta), \theta\in [-r_0,0],$ we have
$$\mu_\theta:= \mu\circ\phi_\theta^{-1}= \mu_0,\ \ \theta\in [-r_0,0].$$ In fact, letting $X_0$ have law $\mu$,  then $X_{-\theta}$ has law $\mu$ as well for any $\theta\in [-r_0,0]$, so that $X_0(\theta)$ has the same distribution as $X_{-\theta}(\theta)=X_0(0)$; that is, $\mu_\theta=\mu_0.$
Moreover, since the equation is non-degenerate, for any $t>0$, the distribution of $X(t)$ has a strictly positive density with respect to the Lebesgue measure. So, $\mu_\theta (\d x)=\rr(x)\d x$ holds for some measurable function $\rr>0$ on $\R^d$ and all $\theta\in [-r_0,0].$

(2) It is well known that when $P_t$ is symmetric in $L^2(\mu)$, the
$L^2$-compactness of $P_t$ for some $t>0$ implies that for all
$t>0.$ This assertion is wrong in the non-symmetric setting. It is
easy to see that in the present framework $P_t$ is not uniformly integrable  (hence, non-compact) on
$L^2(\mu)$ for $t\in [0,r_0]$, since  according to (1)  $\mu_{-r_0}=\mu_{t-r_0}$ has full support on $\R^d$,  and it is obvious that
$$P_t f(\xi)= \E f(X^\xi_t)=  g(\xi(t-r_0)),\ \ \xi\in\C, t\in (0, r_0]  $$ holds for $f(\xi):= g(\xi(-r_0)), g\in \B_b(\R^d).$  Therefore, Theorem \ref{T1.1}
provides a class of   Markov  semigroups which are compact for large
$t$ but not uniformly integrable (hence, non-compact) for small $t\in (0,r_0]$. Moreover, when $r_0=0$, assertions
in Theorem \ref{T1.1} reduce back to the corresponding well known
ones for SDEs without memory.

\

In applications, the following consequence of Theorem \ref{T1.1} is more convenient to use.

\beg{cor}\label{C1.2} Let $k_1,k_2>0$ be two constants such that
\beq \label{C01}  \<Z(x)-Z(y), x-y\>\le -k_1 |x-y|^2,\ \ x,y\in
\R^d,\end{equation} \beq\label{C02}  |b(\xi)-b(\eta)|\le k_2
\|\xi-\eta\|_\infty, \ \ \xi,\eta\in\C. \end{equation} If
\beq\label{C03} k_2^2 \le \ff{2(\ss{k_1^2r_0^2+1}-1)}{r_0^2}
\exp\Big[\ss{k_1^2r_0^2+1}-1-k_1r_0\Big],\end{equation} then
assertions in Theorem $\ref{T1.1}$ hold for
$$\ll:= \ff{r_0}{k_1r_0-1+\ss{k_1^2r_0^2+1}}\bigg(\ff{2(\ss{k_1^2r_0^2+1}-1)}{r_0^2} -k_2^2 \exp\Big[1+k_1r_0-\ss{k_1^2r_0^2+1}\Big]\bigg)>0.$$
\end{cor}

 \

 Next, we consider a semi-linear model which may not satisfy conditions in Theorem \ref{T1.1} and Corollary
 \ref{C1.2}. Let $\R^d\otimes\R^d$ be the set of all real $d\times d$-matrices, and let   $\nu$ be  a $\R^d\otimes \R^d$-valued finite signed measure on
$[-r_0,0]$; that is, $\nu=(\nu_{ij})_{1\le i,j\le d},$ where every $\nu_{ij}$ is a finite signed measure on $[-r_0,0].$ Consider the following semi-linear FSDE
\beq\label{3.1} \d X(t)=
\Big\{\int_{-r_0}^0\nu(\d\theta)X(t+\theta)+b(X_t)\Big\}\d t + \si
\d B(t),
\end{equation}
where  $\si, B(t)$ are   as in \eqref{1.1}, and $b$
satisfies \eqref{C02}. Let
$$\ll_0:= \sup\bigg\{{\rm Re} (\ll):\ \ll \in \mathbb C, {\rm det}\bigg(\ll I_{d\times d}-\int_{-r_0}^0 \e^{\ll s}\nu(\d s)\bigg)=0\bigg\},$$
where $I_{d\times d}\in\R^d\otimes\R^d$ is the unitary matrix.

In particular, when  $\nu=A\dd_0$, where $A\in\R^d\otimes\R^d$ and
  $\dd_0$ is the Dirac measure at point $0$,    equation \eqref{3.1} reduces to the usual semi-linear FSDE:
$$ \d X(t)=
\big\{AX(t)+b(X_t)\big\}\d t + \si
\d B(t), $$ and $\ll_0$ is the largest real part of eigenvalues of $A$.

Let
$\GG(0)=I_{d\times d}, \GG(\theta)=0_{d\times d}$ for $\theta \in
[-r_0,0)$, and $\{\GG(t)\}_{t\ge 0}$ solve the following   equation on $\R^d\otimes\R^d$:
\beq\label{GG} \d \GG(t)= \bigg(\int_{-r_0}^0 \nu(\d\theta)\GG(t+\theta)\bigg)\d t.\end{equation}
According to \cite[Theorem 3.1]{RR}, the unique strong solution
$\{X^\xi(t)\}_{t\ge0}$ of \eqref{3.1}   can be represented by
\begin{equation}\label{a7}
\begin{split}
X^\xi(t)&=\GG(t)\xi(0)+\int_{-r_0}^0\nu(\d \theta)\int_{-r_0}^\theta
\GG(t+\theta-s)\xi(s)\d s\\
&\quad+\int_0^t\GG(t-s)b(X_s^\xi)\d s+\int_0^t \GG(t-s)\si\d B(s).
\end{split}
\end{equation}
In what follows, we assume $\ll_0<0.$ By \cite[Theorem 3.2,
p.271]{HV}, for any $k\in(0,-\ll_0)$, there exists a constant $c_k>0$ such that
\begin{equation}\label{b1}
\|\GG(t)\|\le c_k\e^{-k t},\ \ \ t\ge-r_0,
\end{equation}
where $\|\cdot\|$ denotes the operator norm of a matrix. Obviously, the optimal constant $c_k$ is increasing in $k\in (0, -\ll_0)$.
If, in particular, $\nu=A\dd_0$ for a symmetric $d\times d$-matrix $A$, \eqref{b1} holds for $c_k=1$ and $k\in (0,-\ll_0].$
In general, see Proposition \ref{PP} in the Appendix of the paper for an explicit estimate on $c_k.$

The second main result in this paper is stated as follows.

 \beg{thm}\label{T1.3} Let $P_t$ be the Markov semigroup associated to the equation $\eqref{3.1}$ such that $\nu$ satisfies $\ll_0<0$ and $b$ satisfies $\eqref{C02}$.  If $\ll:=\sup_{k\in(0,-\ll_0)}(k-c_k k_2\e^{kr_0})>0$, where $c_k$ is in $\eqref{b1}$,  then assertions in Theorem $\ref{T1.1}$ hold.  \end{thm}

The following corollary follows immediately from Theorem \ref{T1.3} since when $b=0$ we have $k_2=0$, and $c_k=1$ when $\nu= A\dd_0$ for some symmetric matrix $A$.

\begin{cor}\label{C1.4} In the situation of Theorem $\ref{T1.3}$.\beg{enumerate} \item[$(1)$]
If   $b\equiv0$,  then assertions in Theorem
$\ref{T1.1}$ hold for all $\ll\in(0,-\ll_0).$
\item[$(2)$] Let $\nu =A\dd_0$ for some  symmetric $d\times d$-matrix $A$ with largest eigenvalue $\ll_0<0.$ If $\ll:=\sup_{k\in(0,-\ll_0]}(k-  k_2\e^{kr_0})>0$,  then assertions in Theorem
$\ref{T1.1}$ hold. \end{enumerate}
\end{cor}

 \

To conclude this section, let us compare Theorems \ref{T1.1} and \ref{T1.3}. The framework of Theorem \ref{T1.1} is more general by the generality of
 $Z$. On the other hand, the following example shows that Theorem \ref{T1.3} is not covered by Corollary \ref{C1.2}, the comparable consequence of
 Theorem \ref{T1.1}. Let $r_0=1$,
$\nu(\cdot):=-I_{d\times d}\e^{-1} \dd_{-1}$, and
 $b\equiv0$. Then,   $$\ll_0=\sup\big\{{\rm Re}(\ll):\ \ll\in \mathbb C, \ll+\e^{-\ll-1}=0\big\}=-1<0,$$ so that  Corollary \ref{C1.4} applies to all $\ll\in (0,-1)$; but Corollary  \ref{C1.2} does not apply since $Z=0.$

 \

The next section is devoted to the proofs of Theorem \ref{T1.1} and
Corollary \ref{C1.2}, while Theorem \ref{T1.3} is proved in Section
\ref{sec3}. Finally, in Appendix we present an estimate on $c_k$ in \eqref{b1}.

\section{Proofs of Theorem \ref{T1.1} and Corollary
\ref{C1.2}}\label{sec2} Since \eqref{A1} remains true for
$\bar\ll_1$ in place of $\ll_1$, where $\bar\ll_1\in (0,\ll_1]$ is
such that $\ll= \bar\ll_1-\ll_2 \e^{r_0\bar\ll_1}>0$, we may and do
assume that $\ll= \ll_1-\ll_2 \e^{r_0\ll_1}$.

\beg{lem}  \label{L2.1} If $\ll>0$, then there exist two constants
$c,\vv>0$ such that
$$\E\, \e^{\vv \|X_t^\xi\|_\infty^2} \le \e^{c(1+\|\xi\|_\infty^2)},\ \ t\ge 0, \xi\in \C.$$\end{lem}

\beg{proof} Since in our proof we need to assume in advance that $\E\, \e^{\vv \|X_t^\xi\|_\infty^2}<\infty$ for some $\vv>0$ and all $t\ge 0$, we adopt an approximation argument. Let
$$\tau_n:=\inf\{t\ge 0: \|X_t^\xi\|_\infty\ge n\},\ \ n\ge1.$$ Then $\tau_n\uparrow\infty$.
Consider the FSDE
$$\d X^{(n)}(t) = \{Z(X^{(n)}(t)+b(X^{(n)}_t)\}1_{[0,\tau_n]}(t)\d t -\ll_1 X^{(n)}(t)1_{(\tau_n, \infty)}(t)\d t +\si\d B(t),\ \ X^{(n)}_0=\xi.$$
Then the equation has a unique solution such that $X^{(n)}_t= X^\xi_t$ for $t\le \tau_n.$ Therefore, for any $t>0$,
\beq\label{N1} \lim_{n\to\infty} \|X^{(n)}_t- X^\xi_t\|_\infty=0, \ \   \   a.s.  \end{equation}
Obviously, there exists a constant $C(n)>0$ such that
$$\big\<\{Z(X^{(n)}(t)+b(X^{(n)}_t)\}1_{[0,\tau_n]}(t)  -\ll_1 X^{(n)}(t)1_{(\tau_n, \infty)}(t),\, X^{(n)}(t)\big\>\le C(n) -\ll_1|X^{(n)}(t)|^2.$$
Then it is trivial to see that
\beq\label{N2} \E \e^{\vv_0 \|X_t^{(n)}\|_\infty^2}<\infty,\ \ n\ge 1, t\ge 0\end{equation} holds for some constant $\vv_0>0.$

Next, let  $\xi_0(\theta)=0, \theta\in [-r_0,0].$ By \eqref{A1}, we have
\beg{equation*}\beg{split} 2\<Z(\xi(0))+ b(\xi),\xi(0)\> &\le  2\<Z(\xi(0))+b(\xi) -Z(0)-b(\xi_0),\,\xi(0)\> + |Z(0)+b(\xi_0)|\cdot |\xi(0)|\\
&\le c_0+
\ll_2\|\xi\|_\infty^2-\ll_1'|\xi(0)|^2,\ \ \xi\in \C\end{split}\end{equation*}  for some constants $c_0>0$ and $\ll_1'>0$ such that
$\ll':= \ll_1'-\ll_2 \e^{r_0\ll_1'}>0.$ So, by It\^o's formula,
$$ \d |X^{(n)}(t)|^2 \le \big\{c_1+ \ll_2\|X^{(n)}_t\|_\infty^2-\ll_1'|X^{(n)}(t)|^2 \big\}\d t +\d M(t)$$ holds for
$c_1:= c_0+ \|\si\|_{HS}^2$ and $\d M(t):=2  \<\si\d B(t), X^{(n)}(t)\>.$  This implies
$$\e^{\ll_1' t} |X^{(n)}(t)|^2 \le |\xi(0)|^2 +\int_0^t \e^{\ll_1' s} (c_1 + \ll_2 \|X_s^{(n)}\|_\infty^2) \d s + \int_0^t \e^{\ll_1' s} \d M(s).$$
Let $N(t):= \sup_{s\in [0,t]} \int_0^s\e^{\ll_1' r}\d M(r).$ We
obtain
\beg{equation*}\beg{split} \e^{\ll_1' t} \|X^{(n)}_t\|^2_\8 &\le \e^{r_0\ll_1'} \sup_{\theta\in [-r_0,0]} \e^{\ll_1'(t+\theta)} |X^{(n)}(t+\theta)|^2\\
&\le \e^{\ll_1' r_0} \|\xi\|^2_\infty +\int_0^t \e^{\ll_1' (s+r_0)} (c_1 + \ll_2 \|X_s^{(n)}\|_\infty^2)\d s +\e^{\ll_1' r_0}N(t)\\
&\le c_2(1+\|\xi\|_\infty^2) \e^{\ll_1' t} +\e^{\ll_1' r_0}N(t) +
\ll_2 \e^{\ll_1' r_0} \int_0^t\e^{\ll_1' s} \|X_s^{(n)}\|_\infty^2
\d s\end{split}\end{equation*} for some constant $c_2>0.$  By
Gronwall's inequality, one has
\beg{equation*}\beg{split} \e^{\ll_1' t} \|X_t^{(n)}\|_\infty^2 \le & c_2(1+\|\xi\|_\infty^2)\e^{\ll_1't} +\e^{\ll_1' r_0}N(t)\\
 &+ \ll_2 \e^{\ll_1' r_0}\int_0^t\big\{c_2(1+\|\xi\|_\infty^2)\e^{\ll_1's} +\e^{\ll_1' r_0}N(s)\big\}\exp\big[\ll_2\e^{\ll_1' r_0}(t-s)\big]\d s.\end{split}\end{equation*} Recalling that $\ll':= \ll_1' - \ll_2\e^{\ll_1'r_0}>0$, we arrive at
\beg{equation*}\beg{split}\|X_t^{(n)}\|_\infty^2 &\le c_2(1+\|\xi\|_\infty^2) +\e^{\ll_1' (r_0-t)}N(t)\\
&\qquad  +\ll_2 \e^{\ll_1' r_0}\int_0^t \big\{c_2(1+\|\xi\|_\infty^2)  +\e^{\ll_1' (r_0-s)}N(s)\big\}\e^{-\ll'(t-s)}\d s\\
&\le c_3\big(1+\|\xi\|_\infty^2+\e^{-\ll_1' t}N(t)\big) +c_3 \int_0^t \e^{-\ll_1' s-\ll'(t-s)}N(s)\d s\end{split}\end{equation*}   for some constant $c_3>0.$  Therefore,  for any $\vv\in (0,1)$,
\beq\label{*1} \E\, \e^{\vv \|X_{t }^{(n)}\|_\infty^2} \le \e^{c_3(1+\|\xi\|_\infty^2)} \ss{I_1\times I_2}\end{equation} holds for
\beg{equation*}\beg{split}
&I_1:= \E \exp\bigg[ 2 c_3\vv \int_0^{t } \e^{-\ll_1' s-\ll'(t -s)}N(s)\d s\bigg],\\
&I_2:= \E \exp\big[ 2 c_3\vv\e^{-\ll'_1t}N(t
)\big].\end{split}\end{equation*} To finish the proof, below we
estimate $I_1$ and $I_2$ respectively.

(a) Estimate on $I_1$. To avoid the singularity of reference probability measures   discussed below when $t\to 0$, we   extend the integrals to the larger interval $[-r_0,t].$ Letting $N(s)=0$ for $s\le 0$, by Jensen's inequality for the probability measure
$$\nu(\d s):= \ff {\ll'\e^{\ll'r_0}}{\e^{\ll' r_0}-\e^{-\ll' t}} \e^{-\ll'(t-s)}\d s\ \ {\rm on}\  [-r_0,t], $$ we have
\beg{equation*}\beg{split} &\exp\bigg[ 4 c_3\vv \int_0^{t} \e^{-\ll_1' s-\ll'(t -s)}N(s)\d s\bigg]\\
 &= \exp\bigg[ \ff{4 c_3\vv (\e^{\ll' r_0}-\e^{-\ll' t})}{\ll'\e^{\ll'r_0}}\int_{-r_0}^{t} \e^{-\ll_1' s}N(s)\nu(\d s)\bigg]\\
&\le \int_{-r_0}^t \exp\bigg[ \ff{4 c_3\vv (\e^{\ll' r_0}-\e^{-\ll' t})}{\ll'\e^{\ll'r_0}}  \e^{-\ll_1' s}N(s) \bigg]\nu(\d s)\\
&\le \ff{\ll'\e^{\ll'r_0}}{\e^{\ll' r_0}-1} \int_{-r_0}^t \exp\bigg[
\ff{4 c_3\vv }{\ll'}  \e^{-\ll_1' s}N(s) \bigg]\e^{-\ll'(t-s)}\d
s.\end{split}\end{equation*}
So, by Jensen's inequality and the
Burkhold-Davis-Gundy inequality, there exists a constant $c_4>0$
such that
\beg{equation*}\beg{split}I_1^2 &\le \E\, \exp\bigg[ 4 c_3\vv \int_0^{t } \e^{-\ll_1' s-\ll'(t -s)}N(s)\d s\bigg]\\
&\le  \ff{\ll'\e^{\ll'r_0}}{\e^{\ll' r_0}-1} \int_{-r_0}^t \e^{-\ll'(t-s)}\E \exp\bigg[
\ff{4 c_3\vv }{\ll'}  \e^{-\ll_1' s}N(s) \bigg]\d s\\
&\le c_4 \E \int_{-r_0}^t \e^{-\ll'(t-s)} \bigg(\exp\bigg[c_4 \vv^2 \e^{-2\ll_1' s} \int_0^s \e^{2\ll_1' u}\|X_{u }^{(n)}\|_\infty^2\,\d u\bigg]\bigg)^{\ff 1 2}\d s\\
&\le c_4 \E \int_{-r_0}^t \e^{-\ll'(t-s)} \exp\bigg[c_4 \vv^2   \int_{-r_0}^s \e^{-2\ll_1'(s- u)}\|X_{u }^{(n)}\|_\infty^2\,\d u\bigg]\d s,
\end{split}\end{equation*} where we set $X_s^{(n)}=\xi$ for $s\le 0$.

Now, using   Jensen's inequality as above for the probability measure
 $$\ff{2\ll_1'\e^{\ll'_1r_0}}{\e^{2\ll_1'r_0}-\e^{-2\ll_1' s}} \e^{-2\ll_1' (s-u)}\d u\ \ \text{on}\ [-r_0, s],$$
we arrive at
\beg{equation*}\beg{split} I_1^2&\le c_5 \E\int_{-r_0}^t \e^{-\ll'(t-s)}\d s \int_{-r_0}^s \e^{c_5 \vv^2 \|X_{u }^{(n)}\|_\infty^2-2\ll_1'(s-u)}\d u \\
& =c_5 \int_{-r_0}^t \E\,\e^{c_5 \vv^2 \|X_{u }^{(n)}\|_\infty^2}\d u\int_u^t \e^{-\ll'(t-s)- 2\ll_1' (s-u)}\d s \end{split}\end{equation*}
for some constant $c_5>0.$ Since $\ll_1'\ge \ll'>0$, we have
$$-2\ll_1'(s-u)-\ll'(t-s)\le -\ll' (t-u)-\ll_1'(s-u),$$ so that this implies
\beq\label{*2} \beg{split} I_1^2&\le c_5  \int_{-r_0}^t \e^{-\ll' (t-u)}\E\,\e^{c_5 \vv^2 \|X_{u }^{(n)}\|_\infty^2}\d u\int_u^t \e^{- \ll_1' (s-u)}\d s\\
&\le \ff{c_5}{\ll_1'}   \int_{-r_0}^t\e^{-\ll'(t-u)} \E\,\e^{c_5 \vv^2 \|X_{u }^{(n)}\|_\infty^2}\,\d u.\end{split}\end{equation}

(b) Estimate on $I_2$. A shown in (a), by the
Burkhold-Davis-Gundy inequality and using Jensen's inequality for
the probability measure $$\ff{2\ll_1'\e^{\ll'_1r_0}}{\e^{2\ll_1'
r_0}-\e^{-2\ll_1' t}} \e^{-2\ll_1' (t -s)}\d s \ \ \text{on}\ [-r_0, t],$$ we
conclude  that \beg{equation}\label{*3}\beg{split} I_2^2
&\le \E \exp\bigg[ c_6 \vv^2 \e^{-2\ll'_1t} \int_{-r_0}^{t} \e^{2\ll_1' s} \|X_s^{(n)}\|_\infty^2\d s \bigg]\\
&\le c_7 \E \int_{-r_0}^{t} \e^{c_7\vv^2 \|X_s^{(n)}\|_\infty^2}\e^{-2\ll_1' (t-s)}\d s  \end{split}\end{equation} holds for some constants $c_6,c_7>0.$

Now, combining \eqref{*1}, \eqref{*2} with \eqref{*3}, and taking
$\vv = \vv_0\land \ff 1 {c_5\lor c_7}$, we arrive at
\beg{equation*}\beg{split} &\E \e^{\vv\|X^{(n)}_{t }\|_\infty^2} \le \e^{c_8(1+\|\xi\|_\infty^2)} \bigg(\int_{-r_0}^t \big(\E \e^{\vv \|X_{s }^{(n)}\|_\infty^2}\big)\e^{-\ll' (t-s)}\d s\bigg)^{\ff 1 2} \\
&\le \e^{c_9(1+\|\xi\|_\infty^2)} +\ff {\ll'} 2  \int_{-r_0}^t \big(\E \e^{\vv \|X_{s }^{(n)}\|_\infty^2}\big)\e^{-\ll' (t-s)}\d s \end{split}\end{equation*} for some constants $c_8,c_9>0.$ Equivalently,
$$\e^{\ll't} \E \e^{\vv\|X^{(n)}_{t }\|_\infty^2}\le \e^{c_9(1+\|\xi\|_\infty^2)+\ll't}+ \ff {\ll'} 2 \int_{-r_0}^t \big(\E \e^{\vv \|X_{s }^{(n)}\|_\infty^2}\big)\e^{\ll's}\d s.$$ By \eqref{N2} and $\vv\le\vv_0$, we see that
$$\E \e^{\vv\|X^{(n)}_{t}\|_\infty^2}<\infty,\ \ t\ge 0.$$ Then,
 by Gronwall's inequality,
$$\e^{\ll' t} \E \e^{\vv\|X^{(n)}_{t }\|_\infty^2}\le \e^{c_9(1+\|\xi\|_\infty^2)+\ll't}+
\ff {\ll'} 2\int_{-r_0}^t \e^{c_9(1+\|\xi\|_\infty^2)+\ll's + \ff {\ll'} 2 (t-s)}\d s.$$ Therefore,
$$\E \e^{\vv\|X^{(n)}_{t }\|_\infty^2}\le \e^{c_9(1+\|\xi\|_\infty^2)}+\ff {\ll'} 2\int_{-r_0}^t \e^{c_9(1+\|\xi\|_\infty^2)-  \ff {\ll'} 2 (t-s)}\d s
\le \e^{c(1+\|\xi\|_\infty^2)}$$ for some constant $c>0.$  According to \eqref{N1},  the proof is finished by letting $n\to\infty$.
\end{proof}

\beg{lem}\label{L2.2}  For any $t\ge 0$ and $\xi,\eta\in\C$,
$\|X_t^\xi-X^\eta_t\|_\infty^2 \le \|\xi-\eta\|_\infty^2 \e^{\ll_1 r_0 -\ll t}.$\end{lem}

\beg{proof} By It\^o's formula, we have
$$\d |X^\xi(t)-X^\eta(t)|^2\le \big(\ll_2\|X_t^\xi-X_t^\eta\|_\infty^2-\ll_1 |X^\xi(t)-X^\eta(t)|^2\big)\d t.$$ Then
$$\e^{\ll_1 t}|X^\xi(t)-X^\eta(t)|^2 \le |\xi(0)-\eta(0)|^2 +\ll_2 \int_0^t  \e^{\ll_1 s} \|X_s^\xi-X^\eta_s\|_\infty^2\d s.$$ So,
$$\e^{\ll_1 t}\|X^\xi_t-X^\eta_t\|_\infty^2 \le \e^{r_0\ll_1}\|\xi-\eta\|^2_\infty +\ll_2\e^{r_0\ll_1} \int_0^t  \e^{\ll_1 s} \|X_s^\xi-X^\eta_s\|_\infty^2\d s.$$  Therefore, the proof is finished by Gronwall's inequality since we have assumed that $\ll= \ll_1-\ll_2\e^{r_0\ll_1}.$
\end{proof}

Now, we introduce the dimension-free Harnack inequality in the sense of \cite{W97}.  We are referred to \cite{BWY,ES, WY} for more results on the Harnack inequality of FSDEs.
Since results in these papers do not directly imply the following Lemma \ref{L2.3}, we include a simple proof using coupling by change of measure
 introduced in \cite{ATW}.  By \eqref{A1} and the Lipschitz property of $b$, \eqref{C02} holds for some $k_2\ge 0$ and
$$2\<Z(x)-Z(y),x-y\>\le 2\<b(\xi_y)-b(\xi_x),x-y\> +(\ll_2-\ll_1)|x-y|^2\le -k_1|x-y|^2,\ \ x,y\in\R^d$$
holds for some constant $k_1\in\R$ as required in Lemma \ref{L2.3}, where $\xi_x(\theta)=x, \xi_y(\theta)=y$
for $\theta\in [-r_0,0].$

\beg{lem} \label{L2.3}Let $\eqref{C01}$ and $\eqref{C02}$ hold for some constants $k_1\in\R$ and $k_2\ge 0.$
  Then, for any $p>1,\dd>0$, positive $f\in \B_b(\C),$  and $\xi,\eta\in\C$,
\beg{equation*}\beg{split} \big(P_{t +r_0}f(\xi)\big)^p \le &\big(P_{t+r_0} f^p(\eta)\big)\exp\bigg[\ff{p^2\|\si^{-1}\|^2(1+\dd)}{2(p-1)}\Big\{\ff{2k_1 |\xi(0)-\eta(0)|^2}{ \e^{2k_1t}-1 }\\
 &\quad+\ff{k_2^2}\dd\Big(r_0\|\xi-\eta\|_\8^2+\ff{|\xi(0)-\eta(0)|^2(\e^{4k_1t}-1-4k_1t\e^{2k_1t})}{2k_1(\e^{2k_1t}-1)^2}\Big)\Big\}\bigg].\end{split}\end{equation*}
\end{lem}
\beg{proof} Let $X_t=X^\xi_t$ and $Y(t)$ solve the equation
$$\d Y(s)= \bigg(Z(Y(s)) + b(X_s)+ g(s)1_{[0,\tau)}(s) \cdot \ff{X(s)-Y(s)}{|X(s)-Y(s)|}\bigg)\d s +\si\d B(s),\ \ Y_0=\eta,$$ where
$$\tau:=\inf\{s\ge 0: X(s)=Y(s)\}$$ is the coupling time and $g\in C([0,\infty))$  is to be determined. It is easy to see that this equation has a unique solution up to the coupling time $\tau$. Letting $Y(s)=X(s)$ for $s\ge \tau$, we obtain a solution $Y(s)$ for all $s\ge 0.$ We will then choose $g$ such that $\tau\le t$, i.e. $X_{t+r_0}=Y_{t+r_0}.$
Obviously, we have
$$\d |X(s)-Y(s)|\le -\big\{k_1|X(s)-Y(s)| +g(s)\big\}\d s,\ \ s<\tau.$$ Then
$$  |X(s) -Y(s)|\le |\xi(0)-\eta(0)|\e^{-k_1s}  -\e^{-k_1 s}\int_0^{s} \e^{k_1r}g(r)\d r,\ \ s\le\tau.$$ Taking
\beq\label{CD2}  g(s)= \ff{|\xi(0)-\eta(0)| \e^{k_1s}}{\int_0^t\e^{2k_1s}\d s},\ \ \ s\in [0,t],\end{equation}
we see that
\beq\label{CD3} |X(s) -Y(s)| \le  \ff{|\xi(0)-\eta(0)|(\e^{2k_1 t-k_1 s}-\e^{k_1s})}{\e^{2k_1t}-1},\ \ 0\le s\le \tau.\end{equation}
In particular, this implies   $\tau\le t$ and thus,    $X_{t+r_0}=Y_{t+r_0}$ as required.

Now, let $h(s):= \si^{-1}\big\{1_{[0,\tau)} g(s)\ff{X(s)-Y(s)}{|X(s)-Y(s)|} +
b(X_s)-b(Y_s)\big\}.$ We have
$$\d Y(s)= \big(Z(Y(s)) + b(Y_s)\big)\d s +\si\d \tt B(s),\ \ Y_0=\eta,\ \ s\in [0,t+r_0],$$ where, by the Girsanov theorem,
$$\tt B(s):= B(s)+\int_0^s \<h(u), \d u\>,\ \ s\in [0,t+r_0]$$ is a $d$-dimensional Brownian motion  under the weighted probability measure $\d\Q:= R\d\P$ with
 $$R:= \exp\bigg[-\int_0^{t+r_0}\<h(s), \d B(s)\>-\ff 1 2 \int_0^{t+r_0} |h(s)|^2 \d s\bigg].$$ By the weak uniqueness of the equation and
 $X_{t+r_0}=Y_{t+r_0}$, we have
 \beq\label{NW} P_{t+r_0} f(\eta)= \E[R f(Y_{t+r_0})] = \E[R f(X_{t+r_0})].\end{equation}
 Then, by Jensen's inequality,
\beq\label{HH}(P_{t+r_0} f(\eta))^p =   \big(\E[R
f(X_{t+r_0})]\big)^p\le (P_{t+r_0}f^p(\xi)) \big(\E
R^{\ff{p}{p-1}}\big)^{p-1}.\end{equation} Noting  that    \eqref{CD3} implies
$$\|\xi_t-\eta_t\|_\infty^2\le 1_{[0,r_0]}(s) \|\xi-\eta\|_\infty^2 + 1_{(r_0,r_0+t]}(s) \ff{ (\e^{2k_1t-k_1(s-r_0)}-\e^{k_1(s-r_0)})^2|\xi(0)-\eta(0)|^2}
{(\e^{2k_1t}-1)^2},$$
we obtain from  \eqref{C02} and \eqref{CD2} that
\beg{equation*}\beg{split} |h(s)|^2&\le \|\si^{-1}\|^2 \Big(1_{[0,t]}(s) g(s) +k_2 \|X_s-Y_s\|_\infty\Big)^2\\
& \le \|\si^{-1}\|^2 \Big( 1_{[0,t]}(s)  (1+\dd) g(s)^2
+ (1+\dd^{-1}) k_2^2  \|X_s-Y_s\|_\infty^2\Big)\\
&\le 1_{[0,t]}(s)  \ff{4 k_1^2  \e^{2k_1 s}\|\si^{-1}\|^2 (1+\dd)|\xi(0)-\eta(0)|^2}{(\e^{2k_1t}-1)^2}\\
&\quad + 1_{(r_0,r_0+t]}(s) \ff{\e^{2k_1 s}\|\si^{-1}\|^2 (1+\dd)|\xi(0)-\eta(0)|^2k_2^2(\e^{2k_1t-k_1(s-r_0)}-\e^{k_1(s-r_0)})^2 }{\dd (\e^{2k_1t}-1)^2} \\
&\quad +1_{[0, r_0]}(s) \|\si^{-1}\|^2 k_2^2 (1+\dd^{-1}) \|\xi-\eta\|_\infty^2= :\gg(s).\end{split}\end{equation*}
Therefore,
  \beq\label{NW2}
\beg{split} &\E   R^{\ff p{p-1}}  \le
\e^{\ff{p^2}{2(p-1)^2}\int_0^{t+r_0} \gg(s)\d s }
 \E\e^{-\ff{p}{p-1}\int_0^{t+r_0}\<h(s), \d B(s)\>-\ff {p^2}
{2(p-1)^2} \int_0^{t+r_0} |h(s)|^2 \d
 s} \\
 & =\e^{\ff{p^2}{2(p-1)^2}\int_0^{t+r_0} \gg(s)\d s } \\
 &=\exp\bigg[\ff{p^2\|\si^{-1}\|^2(1+\dd)}{2(p-1)^2}\bigg( \ff{(\e^{4k_1 t}-1-4k_1t \e^{2k_1t})k_2^2|\xi(0)-\eta(0)|^2}{2k_1\dd(\e^{2k_1t}-1)^2}\\
 &\qquad\qquad\qquad\qquad\qquad\qquad + \ff{r_0 k_2^2\|\xi-\eta\|_\infty^2}\dd +  \ff{2k_1|\xi(0)-\eta(0)|^2}{\e^{2k_1t}-1}\bigg)\bigg]. \end{split}\end{equation}
Combining this with \eqref{HH}, we finish the proof.
 \end{proof}

\beg{lem}\label{L2.4} If $\ll>0$, then $P_t$ has a unique invariant
probability measure $\mu$ such that
$$\lim_{t\to\infty} P_t f(\xi)= \mu(f):=\int_\C f\d\mu,\ \ f\in C_b(\C), \xi\in\C.$$\end{lem}

\beg{proof} Let $\scr P(\C)$ be the set of all probability measures on $\C$. Let $W$ be the $L^2$-Wasserstein distance on $\scr P(\C)$ induced by the distance $\rr(\xi,\eta):= 1\land \|\xi-\eta\|_\infty;$ that is,
$$W(\mu_1,\mu_2):= \inf_{\pi\in \C(\mu_1,\mu_2)} \big(\pi(\rr^2)\big)^{\ff 1 2},\ \ \mu_1,\mu_2\in \scr P(\C),$$
where $\C(\mu_1,\mu_2)$ is the set of all couplings of $\mu_1$ and $\mu_2.$ It is well known that $\scr P(\C)$ is a complete metric space with respect to the distance $W$, and the convergence in $W$ is equivalent to the weak convergence, see e.g. \cite[Theorems 5.4 and 5.6]{Chen}. Let $P_t^\xi$ be the distribution of $X_t^\xi$. Then it remains to prove the following two assertions.
\beg{enumerate} \item[(i)] For any $\xi\in \C$, there exists $\mu_\xi\in \scr P(\C)$ such that $\lim_{t\to\infty} W(P_t^\xi,\mu_\xi)=0$.
\item[(ii)] For any $\xi,\eta\in \C$, $\mu_\xi=\mu_\eta$. \end{enumerate}
It is easy to see that (ii) follows from (i) and Lemma \ref{L2.2}. So, we only prove (i). To this end, it suffices to show that when $t\to\infty$,
$\{P_t^\xi\}_{t\ge 0}$ is a Cauchy sequence with respect to $W$.

For any $t_2>t_1>0$, we consider the following equations
\beg{equation*}\beg{split} &\d X(t) =\{Z(X(t))+b(X_t)\}\d t +\si \d B(t),\ \ X_0=\xi, t\in [0,t_2],\\
&\d \bar X(t) =\{Z(\bar X(t))+b(\bar X_t)\}\d t +\si \d B(t),\ \
\bar X_{t_2-t_1}=\xi, t\in [t_2-t_1,t_2].\end{split}\end{equation*}
Then the distributions of $X_{t_2}$ and $\bar X_{t_2}$ are
$P_{t_2}^\xi$ and $P_{t_1}^\xi$ respectively. By \eqref{A1}, we have
$$\d |X(t)-\bar X(t)|^2 \le \big(\ll_2 \|X_t-\bar X_t\|_\infty^2-\ll_1|X(t)-\bar X(t)|^2\big)\d t,\ \ t\in [t_2-t_1,t_2].$$
As in the proof of Lemma \ref{L2.2}, this implies
$$\|X_t-\bar X_t\|_\infty^2 \le \e^{\ll_1 r_0} \|X_{t_2-t_1}-\xi\|_\infty^2 \e^{-\ll (t+t_1-t_2)},\ \ t\in [t_2-t_1,t_2].$$
In particular,
$$\|X_{t_2}-\bar X_{t_2}\|_\infty^2 \le \e^{\ll_1 r_0} \|X_{t_2-t_1}-\xi\|_\infty^2 \e^{-\ll t_1}.$$
By Lemma \ref{L2.1}, we have
$$\E \|X_{t_2-t_1}-\xi\|_\infty^2\le C:=\sup_{t\ge 0} \E \|X_{t}-\xi\|_\infty^2<\infty.$$
Then
$$W(P_{t_1}^\xi,P_{t_2}^\xi)\le \E \{1\land |X_{t_2}-\bar X_{t_2}\|_\infty^2\} \le C \e^{\ll_1 r_0 - \ll t_1}, $$ which goes to zero as $t_1\to\infty$. Therefore, when $t\to\infty$, $\{P_t^\xi\}_{t\ge 0}$ is a Chauchy sequence with respect to $W$.
\end{proof}

\beg{proof}[Proof of Theorem \ref{T1.1}] (a) We first prove that $\|P_t\|_{2\to 4}<\infty$ holds for large enough $t>0.$ Let $f\in \B_b(\C)$
with $\mu(f^2)=1$. By Lemma \ref{L2.3}, for any $t_0>r_0$ there exists a constant $c_0>0$ such that
$$(P_{t_0} f(\xi))^2\le (P_{t_0} f^2(\eta)) \e^{c_0\|\xi-\eta\|^2_\infty},\ \ \xi,\eta\in\C.$$ By the Markov property and Schwartz's inequality,
\beg{equation*}\beg{split} |P_{t+t_0} f(\xi)|^2& = |\E (P_{t_0} f)(X_t^\xi)|^2  \le \Big(\E \ss{ (P_{t_0} f^2 (X_t^\eta))\exp[c_0 \|X^\xi_t-X_t^\eta\|_\infty^2]}\Big)^2\\
&\le (\E (P_{t_0} f^2 (X_t^\eta))  \E \e^{c_0\|X^\xi_t-X_t^\eta\|_\infty^2}= (P_{t+t_0}f^2(\eta)) \E \e^{c_0\|X^\xi_t-X_t^\eta\|_\infty^2}.\end{split}\end{equation*}
Combining this with Lemma \ref{L2.2}, we obtain
$$|P_{t+t_0} f(\xi)|^2\le (P_{t+t_0}f^2(\eta)) \exp\big[c_1\e^{-\ll t}\|\xi-\eta\|_\infty^2\big].$$
Let $r>0$ such that $\mu(B_r)\ge \ff 1 2 $, where
$B_r:=\{\|\cdot\|_\infty<R\}.$ Then
\beg{equation*}\beg{split} &|P_{t+t_0} f(\xi)|^2\exp\big[ -c_1 \e^{-\ll t}( \|\xi\|_\infty +r)^2\big]\le 2 |P_{t+t_0} f(\xi)|^2 \int_{B_r} \exp\big[-c_1\e^{-\ll t}\|\xi-\eta\|_\infty^2\big]\mu(\d\eta)\\
&\le 2    \int_\C P_{t+t_0}f^2(\eta) \mu(\d \eta)
=2.\end{split}\end{equation*} Thus, \beq\label{HP} |P_{t+t_0}
f(\xi)|^4 \le \exp\big[c_2(1+\|\xi\|_\infty^2\e^{-\ll
t})\big],\ \ t\ge 0\end{equation}  holds for some constant $c_2>0.$ On the
other hand, by Lemmas \ref{L2.1} and \ref{L2.4} we have
$$\mu(N\land \e^{\vv \|\cdot\|_\infty^2}) =\lim_{t\to\infty} \E (N\land \e^{\vv \|X^0_t\|_\infty^2})\le \e^{c}<\infty,\ \ N>0$$ for some constant $c>0$. Letting $N\to\infty$ we obtain $\mu(\e^{\vv\|\cdot\|_\infty^2})<\infty$. Therefore, \eqref{HP} implies
$\|P_{t+t_0}\|_{2\to 4} <\infty$ for large enough $t>0$.

(b) To prove Theorem \ref{T1.1}(3), we let $X_t, Y_t$ and $R$ be in
the proof of Lemma \ref{L2.3}.  By \eqref{NW} and $P_{t+r_0} f(\xi)=
\E f(X_{t+r_0})$, we have
$$|P_{t+r_0}f(\xi)-P_{t+r_0}f(\eta)| \le  \E |f(X_{t+t_0})(R-1)| \le \ss{(P_{t+r_0}f^2(\xi))\E(R^2-1)}.$$
Take $p=2$ and $t=t_1>0$ such that $\|P_{t_1+r_0}\|_{2\to 4}<\infty$ according to (a). By \eqref{NW2} there exists a constant $c_1>0$ such that
$\E R^2 \le \e^{c_1\|\xi-\eta\|_\infty^2}.$ So,
\beq\label{NW0} \beg{split} |P_{t_1+r_0}f(\xi)-P_{t_1+r_0}f(\eta)|^2 &\le (P_{t_1+r_0}f^2(\xi))(\e^{c_1\|\xi-\eta\|_\infty^2}-1)\\
&\le (P_{t_1+r_0}f^2(\xi))c_1\|\xi-\eta\|_\infty^2\e^{c_1\|\xi-\eta\|_\infty^2}.\end{split}\end{equation}
Hence, for any $t>0$,
\beg{equation*}\beg{split} &|P_{t+2(t_1+r_0)}f(\xi)-P_{t+2(t_1+r_0)}f(\eta)|^2 \le \big(\E |P_{t_1+r_0}(P_{t_1+r_0}f)(X_t^\xi)-P_{t_1+r_0}(P_{t_1+r_0}f)(X_t^\eta)|\big)^2 \\
&\le \Big( \E\ss {(P_{t_1+r_0}(P_{t_1+r_0}f)^2(X_t^\xi))c_1\|X_t^\xi-X_t^\eta\|_\infty^2\e^{c_1\|X_t^\xi-X_t^\eta\|_\infty^2}}\Big)^2\\
&\le (P_{t+t_1+r_0} (P_{t_1+r_0}f)^2(\xi)) \E \Big[c_1\|X_t^\xi-X_t^\eta\|_\infty^2\e^{c_1\|X_t^\xi-X_t^\eta\|_\infty^2} \Big].\end{split}\end{equation*}
Combining this with Lemma \ref{L2.2}, we arrive at
\beg{equation*}\beg{split} & |P_{t+2(t_1+r_0)}f(\xi)-P_{t+2(t_1+r_0)}f(\eta)|^2\\
 &\le (P_{t+t_1+r_0} (P_{t_1+r_0}f)^2(\xi)) c_2\e^{-\ll t}\|\xi-\eta\|_\infty^2\exp\big[c_2\e^{-\ll t}\|\xi-\eta\|_\infty^2\big]\\
&\le (P_{t+t_1+r_0} (P_{t_1+r_0}f)^2(\xi)) c_3\e^{-\ll t} \exp\Big[\ff \vv 4\|\xi-\eta\|_\infty^2\Big]\end{split}\end{equation*}
for some constants $c_2,c_3>0$ and large enough $t>0$, where $\vv>0$ is such that $\mu(\e^{\vv \|\cdot\|_\infty^2})<\infty$ according to (a).
So,
\beg{equation*}\beg{split} & 2\mu(|P_{t+2(t_1+r_0)}f -\mu(f)|^2) = \int_{\C\times\C} |P_{t+2(t_1+r_0)}f(\xi)-P_{t+2(t_1+r_0)}f(\eta)|^2\mu(\d\xi)\mu(\d \eta)
\\
&\le c_3 \e^{-\ll t} \bigg(\int_{\C} \big\{P_{t+t_1+r_0} (P_{t_1+r_0}f)^2\big\}^2(\xi)\mu(\d\xi)\bigg)^{\ff 1 2} \bigg(\int_{\C\times\C} \exp\Big[\ff \vv 2\|\xi-\eta\|_\infty^2\Big]\mu(\d\xi)\mu(\d\eta)\bigg)^{\ff 1 2} \\
& \le C\e^{-\ll t}\mu(f^2),\ \ t\ge 0\end{split}\end{equation*} holds for some constant $C>0$, since by   Jensen's inequality and that $\mu$ is $P_t$-invariant, we have
 $$\int_{\C} \big\{P_{t+t_1+r_0} (P_{t_1+r_0}f)^2\big\}^2 \d\mu \le \int_\C (P_{t_1+r_0}f)^4\d\mu \le \|P_{t_1+r_0}\|_{2\to 4}^4 \mu(f^2)^2.$$
 Therefore, the assertion in Theorem \ref{T1.1}(3) holds for large enough $t>0$.
Since $P_t$ is contractive in $L^2(\mu)$, it holds for all $t>0$.

(c) We now come back to the proof of Theorem \ref{T1.1}(1). This assertion follows from (a) and Theorem \ref{T1.1}(3) by straightforward calculations. Let $f\in L^2(\mu)$ with $\mu(f^2)=1.$ Let $\hat f= f-\mu(f).$ We have $\mu(P_t \hat f)= \mu(\hat f)=0$. Let $t_0>r_0$ such that
$\|P_{t_0}\|_{2\to 4}<\infty$, we obtain
\beg{equation*}\beg{split}  \mu((P_{t+t_0}f)^4)  & =\mu(f)^4 + 4\mu(f) \mu((P_{t+t_0}\hat f)^3)  + 6 \mu(f)^2 \mu((P_{t+t_0}\hat f)^2) +\mu((P_{t+t_0}\hat f)^4)\\
&\le \mu(f)^4 + 4|\mu(f)|\cdot \|P_{t_0}\|_{2\to 3}^3 \big\{\mu((P_t\hat f)^2)\big\}^{\ff 3 2} \\
&\quad + 6 \mu(f)^2 \mu((P_{t+t_0}\hat f)^2)
+\|P_{t_0}\|_{2\to 4}^4 \mu((P_t\hat f)^2)^2 \\
&\le \mu(f)^4 + c \e^{-\ll t}\big\{|\mu(f)|\big(\mu({\hat f}^2)\big)^{\ff 3 2}  +\mu(f)^2 \mu({\hat f}^2)+ \big(\mu({\hat f}^2)\big)^2\big\}
\end{split}\end{equation*} for some constant $c>0$ according to Theorem \ref{T1.1}(3). Since
$$|\mu(f)| \big(\mu({\hat f}^2)\big)^{\ff 3 2} \le \mu(f)^2 \mu({\hat f}^2) + \big(\mu({\hat f}^2)\big)^2,$$
this implies that for large $t>0$,
$$\mu((P_{t+t_0}f)^4)\le \mu(f)^4 + 2 \mu(f)^2  \mu({\hat f}^2) +  \big(\mu({\hat f}^2)\big)^2=\big\{\mu(f)^2+   \mu({\hat f}^2)\}^2=\mu(f^2)=1.$$

(d)  By e.g. Proposition 3.1(2) in \cite{WY}, the Harnack inequality implies that $P_t$ has a  density with respect to $\mu$ for $t>r_0$. Thus, according to \cite[Theorem 2.3]{Wu}, the hyperboundedness of $P_t$  proved in (a) implies that $P_t$ is compact in $L^2(\mu)$ for large enough $t>0$. Moreover, according to \cite[Proposition 2.3]{W14N}, the hypercontractivity  
implies the desired exponential convergence of $P_t$ in entropy. Hence, Theorem \ref{T1.1}(2) is proved.

(e) Finally, we prove Theorem \ref{T1.1}(4). By the first inequality
in \eqref{NW0}, we have
$$\|\nn_\eta P_{t_1+r_0}f\|^2 (\xi):= \limsup_{s\to 0} \ff {|P_{t_1+r_0} f(\xi+s\eta)- P_{t_1+r_0} f(\xi)|^2}{s^2} \le c_1 \|\eta\|^2_\infty P_{t_1+r_0} f^2(\xi). $$
Thus,
$|P_{t_1+r_0} f(\xi)- P_{t_1+r_0} f(\eta)|^2\le c_1 \|f\|_\infty^2\|\xi-\eta\|_\infty^2.$ Combining this with Lemma \ref{L2.2} and using the Markov property, we obtain
$$|P_{t+t_1+r_0} f(\xi)- P_{t+t_1+r_0} f(\eta)|^2 \le c_1 \|f\|_\infty^2 \E \|X_t^\xi-X_t^\eta\|_\infty^2\le c_2 \e^{-\ll t}\|f\|_\infty^2 $$
for some constants $c_2>0.$ This completes the proof.
\end{proof}

\beg{proof}[Proof of Corollary \ref{C1.2}.] By \eqref{C01} and \eqref{C02}, for any $s>0$ we have
\beg{equation*}\beg{split} & 2\<Z(\xi(0))- Z(\eta(0)) + b(\xi)-b(\eta), \xi(0)-\eta(0)\> \\
&\le -2 k_1 |\xi(0)-\eta(0)|^2 + 2 k_2 \|\xi-\eta\|_\infty\cdot |\xi(0)-\eta(0)| \\
&\le -(2k_1-s)|\xi(0)-\eta(0)|^2 + \ff{k_2^2}s \|\xi-\eta\|_\infty^2.\end{split}\end{equation*}
Let $\ll_1(s)= 2k_1-s, \ll_2(s)= \ff{k_2^2}s.$ Then Theorem \ref{T1.1} applies if  there exists $s\in (0, 2k_1]$ such that
$$\ll_2(s) < \ll_1(s) \e^{-r_0 \ll_1(s)} = (2k_1-s)\e^{-r_0(2k_1-s)};$$ that is,
\beq\label{C04}    k_2^2  < \sup_{s\in (0, 2k_1)}   (2k_1s-s^2)    \e^{-r_0(2k_1-s)},\end{equation}
where the sup is reached at
$$s_0= \ff{k_1 r_0 +\ss{k_1^2r_0^2+1}-1}{r_0},$$ such that \eqref{C04} coincides with \eqref{C03} and Theorem \ref{T1.1} applies to
\beg{equation*}\beg{split} \ll &:= \ll_1(s_0)- \ll_2 (s_0)\e^{r_0\ll_1(s_0)}\\
 &= \ff{r_0}{k_1r_0-1+\ss{k_1^2r_0^2+1}}\bigg(\ff{2(\ss{k_1^2r_0^2+1}-1)}{r_0^2} -k_2^2 \exp\Big[1+k_1r_0-\ss{k_1^2r_0^2+1}\Big]\bigg).\end{split}\end{equation*}
 \end{proof}

\section{Proof of Theorem \ref{T1.3}}

We first recall  the following     Fernique inequality \cite{F} (see also \cite{B}).

 \beg{lem}[Fernique Inequality] \label{L3.1} Let $(X(t))_{t\in D}$ be a family of centered Gaussian random variables on $\R^d$  with
$$  \sup_{t\in D}\E |X(t)|^2\le \si <\infty $$ for some constant $\si>0,$  where $D=\prod_{1\le i\le N} [a_i,b_i]$ is a cube in $\R^N$. Let $\phi\in C([0,\infty])$ be non-decreasing such that $\int_0^\infty \phi(\e^{-r^2})\d r<\infty$ and
$$\E |X(t)-X(s)|^2\le \phi(|t-s|),\ \ s,t\in D.$$ Then there exist constants $C_1,C_2>0$ depending only on $(b_i-a_i)_{1\le i\le N}, N, d, \phi$ and $\si$ such that
$$\P\bigg(\sup_{t\in D} |X(t)|\ge r\bigg)\le C_1 \e^{-C_2 r^2},\ \ r\ge 1.$$\end{lem}

\begin{proof}[Proof of Theorem \ref{T1.3}]\label{sec3} Let $P_t$ be the Markov semigroup associated to the equation $\eqref{3.1}$ such that $\nu$ satisfies $\ll_0<0$ and $b$ satisfies $\eqref{C02}$.
According to the proof of  Lemma \ref{L2.4}, it is easy to see that \eqref{a7} and \eqref{b1}  for some $k\in (0,-\ll_0)$
imply that $P_t$ has  a unique invariant probability measure
$\mu$. Moreover,  by taking $Z=0$ and combining  the linear drift with $b$, we see that Lemma
\ref{L2.3} applies to the present equation for $k_1=0$ and some constant $k_2>0.$  Thus, following the line in the proof of Theorem \ref{T1.1}, we only need to show that   Lemma \ref{L2.1} and Lemma
\ref{L2.2} apply to the equation $\eqref{3.1}$ as well.

\smallskip

 Let $k\in(0, -\ll_0)$
such that
\begin{equation}\label{c1}
\ll:= k-c_kk_2\e^{kr_0}>0.
\end{equation}
It follows from \eqref{C02}, \eqref{a7}  and \eqref{b1} that
\begin{equation*}
\begin{split}
 |X^\xi(t)-X^\eta(t)|&\le\|\GG(t)\|\cdot |\xi(0)-\eta(0)|+\int_{-r_0}^0\Big\|\int_{-r_0}^\theta
\GG(t+\theta-s)\nu(\d \theta)\Big\|\cdot|\xi(s)-\eta(s)|\d s\\
 &\quad+\int_0^t\|\GG(t-s)\|\cdot |b(X_s^\xi)-b(X_s^\eta)|\d s\\
 &\le C_1\e^{-kt}\|\xi-\eta\|_\8+c_kk_2\int_0^t\e^{-k(t-s)} \|X_s^\xi-X_s^\eta\|_\8\d s
 \end{split}
\end{equation*}
for some constant $C_1\ge 1$. Then
\begin{equation*}
\begin{split}
\e^{kt}\|X^\xi_t-X^\eta_t\|_\8&\le \e^{kr_0}\sup_{t-r_0\le s\le
t}(\e^{k s}|X^\xi(s)-X^\eta(s)|)\\
&\le C_1\e^{kr_0}\|\xi-\eta\|_\8+c_kk_2\e^{kr_0}\int_0^t\e^{ks}
\|X_s^\xi-X_s^\eta\|_\8\d s.
\end{split}
\end{equation*}
This, together with   Gronwall's inequality, gives that
\begin{equation}\label{c2}
\begin{split}
\|X^\xi_t-X^\eta_t\|_\8
&\le C_1\e^{kr_0}\|\xi-\eta\|_\8\e^{-\ll t}.
\end{split}
\end{equation}
So, Lemma   \ref{L2.2} applies.

\smallskip

  Next, by
\eqref{C02}, \eqref{a7}  and \eqref{b1},
\begin{equation*}
\begin{split}
\e^{kt}\|X^\xi_t\|_\8 &\le C_2(\|\xi\|_\8+ \e^{kt})
 +c_k\e^{kr_0}k_2\int_0^t\e^{ks}\|X_s^\xi\|_\8\d
 s\\
 &\quad+ \e^{kr_0}\sup_{(t-r_0)^+\le s\le t}\Big(\e^{ks}\Big|\int_0^s\GG(s-r)\si\d
 B(r)\Big|\Big)
 \end{split}
\end{equation*}
holds for some constant  $C_2>0.$ By Gronwall's inequality, this implies
 \begin{equation*}
\begin{split}
\|X^\xi_t\|_\8 &\le
 C_2(\|\xi\|_\8+1) + \e^{kr_0}\sup_{t-r_0\le s\le t} \Big|\int_0^s\GG(s-r)\si\d
 B(r)\Big| \\
 &\quad+ C_2 (\|\xi\|_\infty+1)  \int_0^t\e^{-\ll (t-s)}\d s
 \\
 &\quad+ \e^{kr_0} \int_0^t\bigg(\sup_{(s-r_0)^+\le u\le s} \Big|\int_0^u\GG(u-r)\si\d
 B(r)\Big|\bigg)\e^{-\ll(t-s)}\d s\\
 &\le C_3 (1+\|\xi\|_\infty^2) + C_3  \int_0^t \e^{-\ll(t-s)} \sup_{u\in [-r_0,0]}|Z_{s,u}|\d s
 \end{split}
\end{equation*}  for some constant $C_3>0,$  where
$$Z_{s,u}:= \int_0^{(s+u)^+} \GG(s+u-r)\si\d B(r),\ \ s\ge 0, u\in [-r_0,0].$$ Then, by Jensen's inequality for the probability measure $\ff{\ll}{1-\e^{-\ll t}}\e^{-\ll (t-s)}\d s$ on $[0,t]$, there exists a constant $C_4>0$ such that
\begin{equation}\label{BB}
\begin{split}
\E\e^{\vv\|X^\xi_t\|_\8^2} &\le \e^{\vv C_4(1+\|\xi\|_\8^2)}\E\exp\bigg[\vv
C_4\bigg(\int_0^t \e^{-\ll(t-s)} \sup_{u\in [-r_0,0]}|Z_{s,u}|\d s\bigg)^2\bigg]\\
&\le \e^{\vv C_4(1+\|\xi\|_\8^2)}\ff{\ll}{1-\e^{-\ll t}}\int_0^t \e^{-\ll(t-s)} \bigg(\E \exp\Big[\ff{C_4\vv }\ll \sup_{u\in [-r_0,0]}|Z_{s,u}|^2\Big]\bigg)\d s,
   \ \  \vv>0.
 \end{split}
\end{equation}  It is easy to see from  \eqref{GG} and \eqref{b1}  that
$$\si:= \sup_{s\ge 0, u\in [-r_0,0]} \E |Z_{s,u}|^2<\infty,$$ and there exist constants $c_1,c_2,c_3>0$ such that
\beg{equation*}\beg{split} \E |Z_{s,u}-Z_{s,v}|^2& \le 2 \E\bigg|\int_{(s+v)^+}^{(s+u)^+} \GG(s+u-r)\si\d B(r)\bigg|^2\\
&\quad  + 2 \E
\bigg|\int_0^{(s+v)^+} \big(\GG(s+u-r)-\GG(s+v-r)\big)\si\d B(r)\bigg|^2\\
&\le c_1 |u-v| + 2 \|\si\|^2 \int_0^{(s+v)^+} \|\GG(s+u-r)-\GG(s+v-r)\|^2\d r\\
&\le c_1|u-v|+c_2|u-v|^2 \le c_3 |u-v|, \ \ s\ge 0, -r_0\le v\le u\le 0.\end{split}\end{equation*}   Thus,
by Lemma \ref{L3.1} with $N=1, D= [-r_0,0]$ and $\phi(r)=cr$,
$$C(\vv):= \sup_{s\ge 0} \E \exp\Big[\ff{C_4\vv }\ll \sup_{u\in [-r_0,0]}|Z_{s,u}|^2\Big]<\infty$$  holds for small enough $\vv>0.$ Therefore,
  \eqref{BB} implies the assertion in  Lemma \ref{L2.1}.
\end{proof}

\section{Appendix}

For application of Theorem \ref{T1.3}, we aim to estimate the constant $c_k$ in \eqref{b1}.
Write $\nu=(\nu_{ij})_{1\le i,j\le d}$ for finite signed measures $\nu_{ij}$ on $[-r_0,0].$ Let $|\nu_{ij}|$ be the total variation of $\nu_{ij}$.
For any $\ll>\ll_0$, define
\beg{equation*}\beg{split} &\|\nu\|= \sup_{1\le i\le d} \ss{\sum_{1\le j\le d}|\nu_{ij}|([-r_0,0])^2},\ \
 T_\ll = 2\e^{\ll^- r_0}\|\nu\|,\ \ \ll^-=(-\ll)\lor 0,\\
 &\rr_\ll= \max_{\theta\in [-T_\ll,T_\ll]} \bigg\|\bigg((\ll+\i\theta)I_{d\times d} -\int_{-r_0}^0 \e^{\ll+\i\theta}\nu(\d s)\bigg)^{-1}-(\ll+\i\theta -\ll_0)^{-1}I_{d\times d}\bigg\|.\end{split}\end{equation*}

\begin{prp}\label{PP}
{\rm For any $\ll>\ll_0$,
\begin{equation*}
\|\GG(t)\|\le
\Big\{\ff{(\ll-\ll_0+1)\pi}{\ll-\ll_0} +\ff{4(|\ll_0|+\e^{\ll^- r_0}\|\nu\|)}{T_\ll}+2\rr_\ll
T_\ll\Big\}\e^{\ll t},\ \ t\ge 0.
\end{equation*}
 }
\end{prp}
\begin{proof} For any $z\ne \ll_0$, define
$$Q_z= z I_{d\times d} -\int_{r_0}^0 \e^{z s} \nu(\d s),\ \ G_z= Q_z^{-1} -\ff 1{ z-\ll_0} I_{d\times d}.$$
 We have (see \cite[Theorem 1.5.1]{HV})
\begin{equation}\label{AA0}
\GG(t)=\lim_{T\to\8}\int_{- T}^{
T}Q_{\ll+i\theta}^{-1} \e^{t(\ll+\i \theta)}\d\theta= \lim_{T\to\8}\int_{- T}^{
T}\Big(G_{\ll+i\theta} +\ff{I_{d\times d}}{\ll-\ll_0+\i\theta}\Big) \e^{t(\ll+\i \theta)}\d\theta,\ \ \ll>\ll_0.
\end{equation}
 Obviously, $\big\|\int_{-r_0}^0 \e^{(\ll +\i T)s}\nu(\d s)\big\|\le \e^{r_0 \ll^- }\|v\|$ and
$$\ss{1+\ll^2 T^{-2}}-\ff{\e^{ \ll^-r_0}\|\nu\|}{|T|}\ge \ff 1 2,\ \ |T|\ge
T_\ll.$$
Then
\begin{equation*}
\begin{split}
\|Q_{\ll+\i T}^{-1}\|&\le\ff{1}{\ss{\ll^2+T^2}-\e^{|\ll|r_0}\|\nu\|}\le\ff{2}{|T|},\ \ |T|\ge T_\ll.
\end{split}
\end{equation*}
This yields
\begin{equation*}
\begin{split}
\|G_{\ll+\i T}\| &\le\|Q_{\ll+\i T}^{-1}\|\cdot\bigg\|\ff{\int_{-r_0}^0
\e^{(\ll+\i T) s}\nu(\d
s)-\ll_0I_{d\times d}}{\ll+\i T  -\ll_0}\bigg\|\\
&\le \ff{2(|\ll_0|+\e^{ \ll^-r_0}\|\nu\|)}{|T|\ss{(\ll-\ll_0)^2+T^2}}\le\ff{2(|\ll_0|+\e^{\ll^-r_0}\|\nu\|)}{T^2},~~~|T|\ge
T_\ll.
\end{split}
\end{equation*}
Thus, for any $T\ge T_\ll$,
\begin{equation}\label{*4}
\begin{split}
\int_{T}^{T}\|G_{\ll+\i \theta}\e^{t(\ll+\i \theta)}\|\d\theta&=\int_{-T_\ll}^{T_\ll}\|G_{\ll+\i\theta}\e^{t(\ll+\i \theta)}\|\d\theta
+\int_{|\theta|>T_\ll}\|G_{\ll+\i\theta}\e^{t(\ll+\i \theta)}\|\d\theta\\
&\le 2 \rr_\ll T_\ll \e^{\ll t} + \ff{4(|\ll_0|+\e^{\ll^- r_0}\|\nu\|)\e^{\ll t}}{T_\ll}.
\end{split}
\end{equation}
On the other hand,
\begin{equation*}
\begin{split}
\lim_{T\to\8}\int_{-T}^T
\ff{\e^{t(\ll+\i \theta)}}{\ll-\ll_0+\i\theta}\,\d\theta &=i\e^{\ll
t}\lim_{T\to\8}\int_{-
T}^{T}\ff{(\ll-\ll_0)\e^{it\theta}}{(\ll-\ll_0)^2+\theta^2}\d\theta-\e^{\ll
t}\lim_{T\to\8}\int_{-
T}^{T}\ff{\theta\e^{it\theta}}{(\ll-\ll_0)^2+\theta^2}\d\theta\\
&=:\Theta_1+\Theta_2.
\end{split}
\end{equation*}
It is easy to see that
\begin{equation*}
\|\Theta_1\|\le\ff{2\e^{\ll
t}}{\ll-\ll_0}\lim_{T\to\8}\arctan\Big(\ff{\theta}{\ll-\ll_0}\Big)\Big|_0^T=\ff{\pi\e^{\ll
t}}{\ll-\ll_0}.
\end{equation*}
Moreover, by the  residue theorem,
\begin{equation*}
\begin{split}
\|\Theta_2\|&=\Big|-2\pi\e^{\ll t}
i\mbox{Res}\Big[\ff{z\e^{itz}}{(\ll-\ll_0)^2+z^2},(\ll-\ll_0)i\Big]\Big|\\
&=\Big|-2\pi\e^{\ll t}
i\lim_{z\to(\ll-\ll_0)i}(z-(\ll-\ll_0)i)\times\ff{z\e^{itz}}{(\ll-\ll_0)^2+z^2}\Big|\\
&=\Big|-2\pi\e^{\ll t}
i\lim_{z\to(\ll-\ll_0)i}\ff{z\e^{itz}}{2(\ll-\ll_0)i}\Big|\\
&=\Big|-2\pi\e^{\ll t}
i\ff{(\ll-\ll_0)i\e^{-t(\ll-\ll_0)}}{2(\ll-\ll_0)i}\Big|\\
 &=\pi\e^{\ll_0t}
 \le\pi\e^{\ll t}.
\end{split}
\end{equation*}
Hence, we arrive at
\begin{equation}\label{*5}
\bigg|\lim_{T\to\8}\int_{  -T}^T
 \ff{\e^{t(\ll+\i \theta)}}{\ll-\ll_0+\i \theta }\d\theta\bigg|\le\ff{(\ll-\ll_0+1)\pi\e^{\ll t}}{\ll-\ll_0}.
 \end{equation}
Combing this with  \eqref{*4} and \eqref{AA0}, we finish the proof.

\end{proof}

\end{document}